\documentclass[a4paper,11pt]{article}
\usepackage[numbers,sort&compress,longnamesfirst]{natbib}
\usepackage{amsmath,amsthm,amsfonts,amssymb,graphicx,mathtools,iwona,float}
\usepackage[dvipsnames]{xcolor}
\usepackage[colorlinks=true,citecolor=black,linkcolor=black,urlcolor=MidnightBlue]{hyperref}
\usepackage[margin=34mm]{geometry}
\usepackage{rsfso}
\usepackage[noabbrev,capitalise]{cleveref}
\makeatletter
\def\NAT@spacechar{~}
\makeatother
\theoremstyle{plain}
\newtheorem{thm}{Theorem}
\newtheorem{lem}[thm]{Lemma}
\newtheorem{cor}[thm]{Corollary}

\newtheorem{conj}[thm]{Conjecture}
\newtheorem{claim}{Claim}
\newtheorem*{claim*}{Claim}
\crefname{claim}{Claim}{Claims}
\crefname{lem}{Lemma}{Lemmas}
\crefname{thm}{Theorem}{Theorems}
\crefname{prop}{Proposition}{Propositions}
\crefname{conj}{Conjecture}{Conjectures}
\renewcommand{\baselinestretch}{1.08}
\usepackage[hang]{footmisc}

\renewcommand{\thefootnote}{\fnsymbol{footnote}}	
\allowdisplaybreaks
\newcommand\DateFootnote{
\begingroup
\renewcommand\thefootnote{}
\footnote{\today}
\setcounter{footnote}{0}
\vspace*{-3ex}
\endgroup}
\makeatletter
\renewcommand\section{\@startsection {section}{1}{\z@}{-3ex \@plus -1ex \@minus -.2ex}{2ex \@plus.2ex}{\normalfont\large\bfseries}}
\renewcommand\subsection{\@startsection{subsection}{2}{\z@}{-2.5ex\@plus -1ex \@minus -.2ex}{1.5ex \@plus .2ex}{\normalfont\normalsize\bfseries}}
\renewcommand\subsubsection{\@startsection{subsubsection}{3}{\z@}{-2ex\@plus -1ex \@minus -.2ex}{1ex \@plus .2ex}{\normalfont\normalsize\bfseries}}
 \renewcommand\paragraph{\@startsection{paragraph}{4}{\z@}{1.5ex \@plus.5ex \@minus.2ex}{-1em}{\normalfont\normalsize\bfseries}}
\renewcommand\subparagraph{\@startsection{subparagraph}{5}{\parindent}  {1.5ex \@plus.5ex \@minus .2ex}  {-1em} {\normalfont\normalsize\bfseries}}
\makeatother
\usepackage[shortlabels]{enumitem}
\setlist[itemize]{topsep=0ex,itemsep=0ex,parsep=0.4ex}
\setlist[enumerate]{topsep=0ex,itemsep=0ex,parsep=0.4ex}
\newcommand{\defn}[1]{\textcolor{Maroon}{\emph{#1}}}
\renewcommand{\thefootnote}{\fnsymbol{footnote}}	

\renewcommand{\geq}{\geqslant}
\renewcommand{\leq}{\leqslant}
\DeclareMathOperator{\td}{td}
\DeclareMathOperator{\tcn}{tcn}
\DeclareMathOperator{\ctd}{\overline{td}}
\DeclareMathOperator{\dist}{dist}
\newcommand{\bigchi}{\raisebox{1.55pt}{\scalebox{1.25}{\ensuremath\chi}}}
\newcommand{\cchi}{\bigchi_{\star}\hspace*{-0.2ex}}
\newcommand{\dchi}{\bigchi\hspace*{-0.1ex}_{\Delta}\hspace*{-0.3ex}}
\newcommand{\cfchi}{\bigchi^f_{\star}\hspace*{-0.1ex}}
\newcommand{\dfchi}{\bigchi^f_{\Delta}\hspace*{-0.1ex}}
\newcommand{\N}{\mathbb{N}}
\newcommand{\GG}{\mathcal{G}}

\newcommand{\MM}{\mathcal{M}}

\newcommand{\la}{\langle}
\newcommand{\ra}{\rangle}
\newcommand{\eps}{\varepsilon}
 
\newcommand{\mc}[1]{\mathcal{#1}} 
\newcommand{\bb}[1]{\mathbb{#1}}

\newcommand{\Closure}[1]{\ensuremath{C\langle{#1}\rangle}}
\newcommand{\WeakClosure}[1]{\ensuremath{W\langle{#1}\rangle}}

\begin{document}

{\Large\bfseries\boldmath\scshape Clustered Colouring of Graph Classes\\ with Bounded Treedepth or Pathwidth}

\medskip
Sergey Norin\footnotemark[2] \quad
Alex Scott\footnotemark[3] \quad
David R. Wood\footnotemark[5] 

\DateFootnote

\footnotetext[2]{Department of Mathematics and Statistics, McGill University, Montr\'eal, Canada (\texttt{snorin@math.mcgill.ca}). Supported by NSERC grant 418520.}

\footnotetext[3]{Mathematical Institute,  University of Oxford, Oxford, U.K.\ (\texttt{scott@maths.ox.ac.uk}).}

\footnotetext[5]{School of Mathematics, Monash University, Melbourne, Australia (\texttt{david.wood@monash.edu}). \\
Supported by the Australian Research Council.}

\emph{Abstract.} The \defn{clustered chromatic number} of a class of graphs is the minimum integer $k$ such that for some integer $c$ every graph in the class is $k$-colourable with monochromatic components of size at most $c$. We determine the clustered chromatic number of any minor-closed class with bounded treedepth, and prove a best possible upper bound on the  clustered chromatic number of any minor-closed class with bounded pathwidth. As a consequence, we determine the fractional clustered chromatic number of every minor-closed class.

\bigskip
\bigskip

\hrule

\bigskip

\renewcommand{\thefootnote}{\arabic{footnote}}

\section{Introduction}

This paper studies improper vertex colourings of graphs with bounded monochromatic degree or bounded monochromatic component size. This topic has been extensively studied recently~\citep{CGJ97,Archdeacon87,OOW19,NSSW19,LO17,EKKOS15,vdHW18,KO19,DN17,MRW17,HST03,EJ14,ADOV03,BK18,CE19,Kawa08,KM07,LW1,LW2,LW3,DEMWW22}; see \citep{WoodSurvey} for a survey. 

A \defn{$k$-colouring} of a graph $G$ is a function that assigns one of $k$ colours to each vertex of $G$. In a coloured graph, a \defn{monochromatic component} is a connected component of the subgraph induced by all the vertices of one colour. 

A colouring has \defn{defect} $d$ if each monochromatic component has maximum degree at most $d$. The \defn{defective chromatic number} of a graph class $\GG$, denoted by $\dchi(\GG)$, is the minimum integer $k$ such that, for some integer $d$, every graph in  $\GG$ is $k$-colourable with defect $d$. 

A colouring has \defn{clustering} $c$ if each monochromatic component has at most $c$ vertices. The \defn{clustered chromatic number} of a graph class $\GG$, denoted by $\cchi(\GG)$, is the minimum integer $k$ such that, for some integer $c$, every graph in $\GG$ has a $k$-colouring with clustering $c$. We shall consider such colourings, where the goal is to minimise the number of colours, without optimising the clustering value. 

Every colouring of a graph with clustering $c$ has defect $c-1$. Thus $ \dchi(\GG) \leq \cchi(\GG)$ for every class $\GG$.

The following is a well-known and important example in defective and clustered graph colouring. Let $T$ be a rooted tree. The \defn{depth} of $T$ is the maximum number of vertices on a root--to--leaf path in $T$. The \defn{closure} of $T$ is obtained from $T$ by adding an edge between every ancestor and descendant in $T$. For $h,k\geq 1$, let \Closure{h,k} be the closure of the complete $k$-ary tree of depth $h$, as illustrated in \cref{StandardExample}. 

\begin{figure}[h]
{\centering\includegraphics[width=\textwidth]{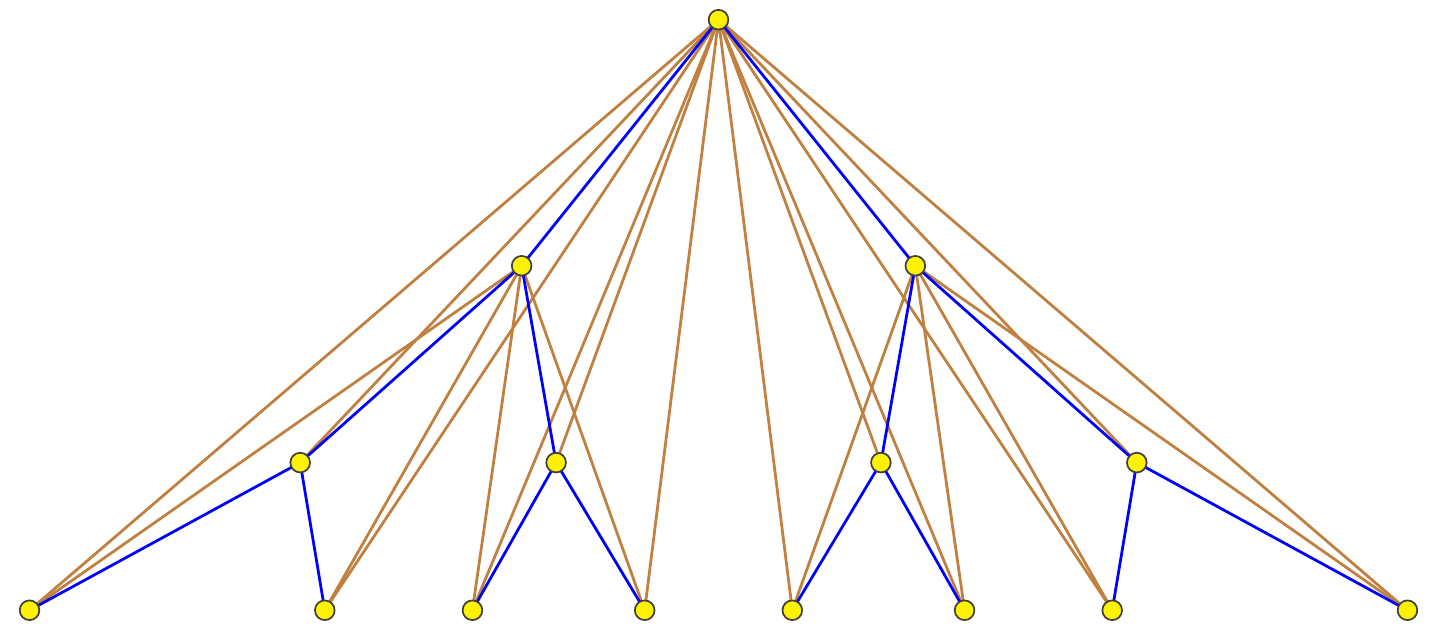}}
\caption{The standard example \Closure{4,2}.
\label{StandardExample}}
\end{figure}

It is well known and easily proved (see \citep{WoodSurvey}) that there is no $(h-1)$-colouring of \Closure{h,k} with defect $k-1$, which implies there is no $(h-1)$-colouring of \Closure{h,k} with clustering $k$. This says that if a graph class $\GG$ includes \Closure{h,k} for all $k$, then the defective chromatic number and the clustered chromatic number are at least $h$. Put another way, define the \defn{tree-closure-number} of a graph class $\GG$ to be 
$$\tcn(\GG) :=
 \min\{ h: \exists k\, \Closure{h,k} \not\in \GG \} = 
 \max\{ h: \forall k\, \Closure{h,k} \in \GG \} +1;$$ 
then 
$$ \cchi(\GG) \geq \dchi(\GG) \geq \tcn(\GG) - 1.$$

Our main result, \cref{Main} below, establishes a converse result for minor-closed classes with bounded treedepth. First we explain these terms. A graph $H$ is a \defn{minor} of a graph $G$ if a graph isomorphic to $H$ can be obtained from some subgraph of $G$ by contracting edges. A class of graphs $\MM$ is \defn{minor-closed} if for every graph $G\in\MM$ every minor of $G$ is in $\MM$, and $\MM$ is \defn{proper} minor-closed if, in addition, some graph is not in $\MM$. The \defn{connected treedepth} of a graph $H$, denoted by $\ctd(H)$, is the minimum depth of a rooted tree $T$ such that $H$ is a subgraph of the closure of $T$. This definition is a variant of the more commonly used definition of the \defn{treedepth} of $H$, denoted by $\td(H)$, which equals the maximum connected treedepth of the connected components of $H$. (See \citep{Sparsity} for background on treedepth.)\  If $H$ is connected, then $\td(H)=\ctd(H)$. In fact, $\td(H)=\ctd(H)$ unless $H$ has two connected components $H_1$ and $H_2$ with $\td(H_1)=\td(H_2)=\td(H)$, in which case $\ctd(H)=\td(H)+1$. It is convenient to work with connected treedepth to avoid this distinction. A class of graphs has \defn{bounded treedepth} if there exists a constant $c$ such that every graph in the class has treedepth at most $c$. 

\begin{thm}
\label{Main}
For every minor-closed class $\GG$ with bounded treedepth, 
$$ \dchi(\GG) = \cchi(\GG) = \tcn(\GG) -1.$$
\end{thm}

Our second result concerns pathwidth. A \defn{path-decomposition} of a graph $G$ consists of a sequence $(B_1,\dots,B_n)$, where each $B_i$ is a subset of $V(G)$ called a \defn{bag}, such that for every vertex $v\in V(G)$, the set $\{i\in[1,n]: v\in B_i\}$ is an interval, and for every edge $vw\in E(G)$ there is  a bag $B_i$ containing both $v$ and $w$. Here $[a,b]:=\{a,a+1,\dots,b\}$. The \defn{width} of a path decomposition $(B_1,\dots,B_n)$ is $\max\{|B_i|: i\in[1,n] \}-1$. The \defn{pathwidth} of a graph $G$ is the minimum width of a path-decomposition of $G$. Note that paths (and more generally caterpillars) have pathwidth 1. A class of graphs has \defn{bounded pathwidth} if there exists a constant $c$ such that every graph in the class has pathwidth at most $c$. 

\begin{thm}
\label{MainMain}
For every minor-closed class $\GG$ with bounded pathwidth, 
$$ \dchi(\GG) \leq \cchi(\GG) \leq 2 \tcn( \GG )-2.$$
\end{thm}

\cref{Main,MainMain} are respectively proved in \cref{Treedepth,Pathwidth}. These results are best possible and partially resolve a number of conjectures from the literature, as we now explain. 

%First, \citet[Lemma~13]{NSSW19} constructed, for each integer $h\geq 2$, a minor-closed class of graphs $\GG_h$ such that $\Closure{h,3}\not\in\GG_h$ and $\cchi(\GG_h) \geq 2h-2$. It is easily seen that every graph $\GG_h$ has pathwidth at most $2h-3$. Thus the upper bound on $\cchi(\GG)$ in \cref{MainMain} is best possible. 

\citet{OOW19} studied the defective chromatic number of minor-closed classes. For a graph $H$, let $\MM_H$ be the class of $H$-minor-free graphs (that is, not containing $H$ as a minor). \citet{OOW19} proved the lower bound, $\dchi(\MM_H) \geq \ctd(H)-1$ and conjectured that equality holds. 

\begin{conj}[\citep{OOW19}] 
\label{tdH}
For every graph $H$, 
$$\dchi(\MM_H)=\ctd(H)-1.$$ 
\end{conj}

\cref{tdH} is known to hold in some special cases. \citet{EKKOS15} proved it if $H=K_t$; that is, $\dchi(\MM_{K_t})=t-1$, which can be thought of as a defective version of Hadwiger's Conjecture; see \citep{vdHW18} for an improved bound on the defect in this case. \citet{OOW19} proved \cref{tdH} if $\ctd(H)\leq 3$ or if $H$ is a complete bipartite graph. In particular, $\dchi(\MM_{K_{s,t}})=\min\{s,t\}$. 

\citet{NSSW19} studied the clustered chromatic number of minor-closed classes. They showed that for each $k\geq 2$, there is a graph $H$ with treedepth $k$ and connected treedepth $k$ such that $\cchi(\MM_{H}) \geq 2k-2$. Their proof in fact constructs a set $\mathcal{X}$ of graphs in $\MM_H$ with bounded pathwidth (at most $2k-3$ to be precise) such that $\cchi(\mathcal{X}) \geq 2k-2$. Thus the upper bound on $\cchi(\GG)$ in \cref{MainMain} is best possible. 

\citet{NSSW19} conjectured the following converse upper bound (analogous to \cref{tdH}):

\begin{conj}[\citep{NSSW19}]
	\label{tdConjecture}
	For every graph $H$,  $$\cchi(\MM_H)\leq 2\ctd(H)-2.$$ 
\end{conj}

While \cref{tdH,tdConjecture} remain open, \citet{NSSW19} showed in the following theorem that $\dchi(\MM_H)$  and $\cchi(\MM_H)$ are controlled by the treedepth of $H$:

\begin{thm}[\citep{NSSW19}]
\label{Previous}
For every graph $H$,  $\cchi(\MM_H)$ is tied to the (connected) treedepth of $H$. In particular, 
$$\ctd(H)-1 \leq \cchi(\MM_H) \leq 2^{\ctd(H)+1}-4.$$
\end{thm}

\cref{Main} gives a much more precise bound than \cref{Previous} under the extra assumption of bounded treedepth. 

Our third main result concerns fractional colourings. For real $t \geq 1$, a graph $G$ is \defn{fractionally $t$-colourable with clustering $c$} if there exist $Y_1,Y_2, \ldots, Y_s \subseteq V(G)$ and $\alpha_1,\ldots,\alpha_s \in [0,1]$ such that\footnote{If $c=1$, then this corresponds to a (proper) fractional $t$-colouring, and if the $\alpha_i$ are integral, then this yields a $t$-colouring with clustering $c$.}:
\begin{itemize}
	\item Every component of $G[Y_i]$ has at most $c$ vertices,
	\item $\sum_{i=1}^s \alpha_i \leq t$,
	\item $\sum_{i : v \in Y_i}\alpha_i \geq 1$ for every $v \in V(G)$.
\end{itemize} 

The \defn{fractional clustered chromatic number $\cfchi(\GG)$} of a graph class $\GG$ is the infimum of $t>0$ such that there exists $c=c(t,\GG)$ such that every $G \in \GG$ is fractionally $t$-colourable with clustering $c$. 

\defn{Fractionally $t$-colourable with defect $d$} and \defn{fractional defective chromatic number} $\dfchi(\GG)$ are defined in exactly the same way, except the condition on the component size of $G[Y_i]$ is replaced by ``the maximum degree of $G[Y_i]$ is at most $d$''. 

The following theorem determines the fractional clustered chromatic number and fractional defective chromatic number of any proper minor-closed class.

\begin{thm}
\label{FractionalChromaticNumber}
For every proper minor-closed class $\GG$, 
$$ \dfchi(\GG)= \cfchi(\GG) = \tcn(\GG) -1.$$ 
\end{thm}

This result is proved in \cref{FractionalColouring}. 

We now give an interesting example of \cref{FractionalChromaticNumber}.

\begin{cor}
	\label{Surface}
For every surface $\Sigma$, if $\GG_\Sigma$ is the class of graphs embeddable in $\Sigma$, then 
$$ \dfchi(\GG_\Sigma)= \cfchi(\GG_\Sigma) = 3.$$ 
\end{cor}

\begin{proof}
Note that $\Closure{3,k}$ is planar for all $k$. Thus $\tcn(\GG_\Sigma)\geq 4$. 
Say $\Sigma$ has Euler genus $g$. It follows from Euler's formula that $K_{3,2g+3}\not\in \GG_\Sigma$. Since $K_{3,2g+3} \subseteq \Closure{4,2g+3}$, we have $\Closure{4,2g+3} \not\in \GG_\Sigma$. Thus $\tcn(\GG_\Sigma)= 4$. The result  follows from \cref{FractionalChromaticNumber}. 
\end{proof}

In contrast to \cref{Surface}, \citet{DN17} proved that $\cchi(\GG_\Sigma)=4$. 
Note that \citet{Archdeacon87} proved that $\dchi(\GG_\Sigma)=3$; see \citep{CGJ97} for an improved bound on the defect. 

%%%%%%%%%%%%
\section{Treedepth}
\label{Treedepth}

Say $G$ is a subgraph of the closure of some rooted tree $T$. For  each vertex $v\in V(T)$, let $T_v$ be the maximal subtree of $T$ rooted at $v$ (consisting of $v$ and all its descendants), and let $G[T_v]$ be the subgraph of $G$ induced by $V(T_v)$.

The \defn{weak closure} of a rooted tree $T$ is the graph $G$ with vertex set $V(T)$, where two vertices $v,w\in V(T)$ are adjacent in $G$ whenever $v$ is a leaf of $T$ and $w$ is an ancestor of $v$ in $T$. As illustrated in \cref{WeakClosure}, let \WeakClosure{h,k} be the weak closure of the complete $k$-ary tree of height $h$. 

\begin{figure}[h]
	{\centering\includegraphics[width=\textwidth]{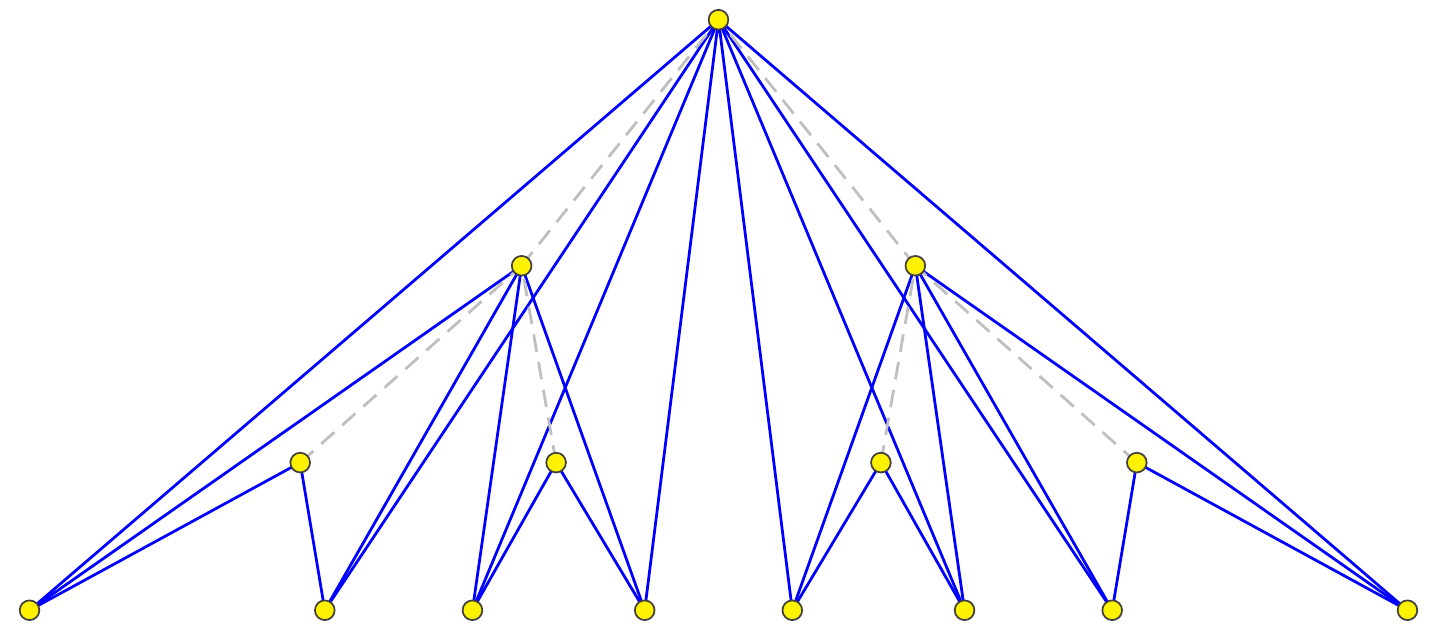}}
	\caption{The weak closure \WeakClosure{4,2}.
		\label{WeakClosure}}
\end{figure}

Note that \WeakClosure{h,k} is a proper subgraph of \Closure{h,k} for $h\geq 3$. On the other hand, \citet{NSSW19} showed that \WeakClosure{h,k} contains \Closure{h,k-1} as a minor
for all $h,k\geq 2$. Therefore \cref{Main} is an immediate consequence of the following lemma.

\begin{lem}
\label{UpperBound}
For all $d,k,h\in\N$ there exists $c=c(d,k,h)\in \N $ such that for every graph $G$ with treedepth at most $d$, either $G$ contains a \WeakClosure{h,k}-minor or $G$ is $(h-1)$-colourable with clustering $c$. 
\end{lem}

\begin{proof}
Throughout this proof, $d$, $k$ and $h$ are fixed, and we make no attempt to optimise $c$. 

We may assume that $G$ is connected. So $G$ is a subgraph of the closure of some rooted tree of depth at most $d$. Choose a tree $T$ of depth at most $d$ rooted at some vertex $r$, such that $G$ is a subgraph of the closure of $T$, and subject to this, $\sum_{v\in V(T)} \dist_T(v,r)$ is minimal. Suppose that $G[T_v]$ is disconnected for some vertex $v$ in $T$. Choose such a vertex $v$ at maximum distance from $r$. Since $G$ is connected, $v\neq r$. By the choice of $v$, for each child $w$ of $v$, the subgraph $G[T_w]$ is connected. Thus, for some child $w$ of $v$, there is no edge in $G$ joining $v$ and $G[T_w]$. Let $u$ be the parent of $v$. Let $T'$ be obtained from $T$ by deleting the edge $vw$ and adding the edge $uw$, so that $w$ is a child of $u$ in $T'$. Note that $G$ is a subgraph of the closure of $T'$ (since $v$ has no neighbour in $G[T_w]$). Moreover, $\dist_{T'}(x,r) = \dist_T(x,r)-1$ for every vertex $x\in V(T_w)$, and $\dist_{T'}(y,r) = \dist_T(y,r)$ for every vertex $y\in V(T)\setminus V(T_w)$. Hence
$\sum_{v\in V(T')} \dist_{T'}(v,r)<\sum_{v\in V(T)} \dist_T(v,r)$, which contradicts our choice of $T$. Therefore $G[T_v]$ is connected for every vertex $v$ of $T$. 

Consider each vertex $v\in V(T)$. Define the \defn{level} $\ell(v):= \dist_T(r,v) \in[0,d-1]$. Let $T^+_v$ be the subtree of $T$ consisting of $T_v$ plus the $vr$-path in $T$, and let $G[T^+_v]$ be the subgraph of $G$ induced by $V(T^+_v)$. For a subtree $X$ of $T$ rooted at vertex $v$, define the \defn{level} $\ell(X):=\ell(v)$. 

% Step 2
A \defn{ranked graph} (for fixed $d$) is a triple $(H,L,\preceq)$ where:
\begin{itemize}
\item $H$ is a graph,
\item $L:V(H)\rightarrow [0,d-1]$ is a function, 
\item $\preceq$ is a partial order on $V(H)$ such that $L(v) < L(w)$ whenever $v \prec w$.
\end{itemize}
Here and throughout this proof, $v\prec w$ means that $v\preceq w$ and $v\neq w$. 
%\referee{Page 6, Line 11: The deﬁnition here should be more careful, as $\prec$ is reflexive. So it should be “$L(v) < L(w)$ whenever $v\prec w$ and $v = w$”.} 
Up to isomorphism, the number of ranked graphs on $n$ vertices is at most $2^{\binom{n}{2}}\,d^n\, 3^{\binom{n}{2}}$. For a vertex $v$ of $T$, a ranked graph $(H,L,\preceq)$ is said to be \defn{contained in} $G[T^+_v]$ if there is an isomorphism $\phi$ from $H$ to some subgraph of $G[T^+_v]$ such that:
\begin{enumerate}[(A)]
	\item for each vertex $v\in V(H)$ we have $L(v)=\ell(\phi(v))$, and 
	\item for all distinct vertices $v,w\in V(H)$ we have that $v\prec w$ if and only if $\phi(v)$ is an ancestor of $\phi(w)$ in $T$. 
\end{enumerate}

Say $(H,L,\preceq)$ is a ranked graph and $i\in[0,d-1]$. Below we define the \defn{$i$-splice} of $(H,L,\preceq)$ to be a particular ranked graph $(H',L',\preceq')$, which (intuitively speaking) is obtained from $(H,L,\preceq)$ by copying $k$ times the subgraph of $H$ induced by the vertices $v$ with $L(v)>i$. Formally, let
\begin{align*}
V(H') := & \{ (v,0) : v \in V(H), L(v) \in [0,i] \} \;\cup \\
 & \{ (v,j) : v \in V(H), L(v) \in[i+1,d], j \in [1,k] \}.\\ 
E(H') := & \{ (v,0)(w,0) : vw \in E(H), L(v) \in[0,i], L(w) \in [0,i] \} \;\cup\\
& \{ (v,0)(w,j) : vw \in E(H), L(v) \in [0,i], L(w) \in[i+1,d], j \in[1,k] \} \; \cup\\
& \{ (v,j)(w,j) : vw \in E(H), L(v)\in[i+1,d], L(w) \in[i+1,d], j \in[1,k] \}.
\end{align*}
Define $L'((v,j)) := L(v)$ for every vertex $(v,j)\in V(H')$.
Now define the following partial order $\preceq'$ on $V(H')$:
\begin{itemize}
	\item $(v,j)\preceq' (v,j)$ for all $(v,j)\in V(H')$;
	\item  if $v \prec w$ and $L(v),L(w) \in[0,i]$, then $(v,0) \prec' (w,0)$;
	\item  if $v \prec w$ and $L(v)\in[0,i]$ and $L(w) \in [i+1,d]$, then $(v,0) \prec' (w,j)$ for all $j\in[1,k]$; and
	\item  if $v \prec w$ and $L(v),L(w) \in [i+1,d]$, then $(v,j) \prec' (w,j)$ for all $j\in[1,k]$.
\end{itemize}
Note that if $(v,a)\prec' (w,b)$, then $a\leq b$ and $v\prec w$ (implying $(L(v)<L(w)$). 
It follows that $\prec'$ is a partial order on $V(H')$ such that $L'((v,a)) < L'((w,b))$ whenever $(v,a) \prec' (w,b)$.
Thus $(H',L',\preceq')$ is a ranked graph. 

For $\ell\in[0,d-1]$, let $$N_\ell:= (d+1)(h-1)(k+1)^{d-1-\ell}.$$ 
For each vertex $v$ of $T$, define the \defn{profile} of $v$ to be the set of all ranked graphs $(H,L,\preceq)$ contained in $G[T_v^+]$  such that $|V(H)| \leq N_{\ell(v)}$. Note that if $v$ is a descendant of $u$, then the profile of $v$ is a subset of the profile of $u$. For $\ell\in[0,d-1]$, if $N=N_\ell$ then let 
$$M_\ell := 2^{ 2^{\binom{N}{2}}\,d^N\, 3^{\binom{N}{2}} }.$$
Then there are at most $M_\ell$ possible profiles of a vertex at level $\ell$.
 
We now partition $V(T)$ into subtrees. Each subtree is called a \defn{group}. 
(At the end of the proof, vertices in a single group will be assigned the same colour.)\ We assign vertices to groups in non-increasing order of their distance from the root. 
Initialise this process by placing each leaf $v$ of $T$ into a singleton group.
We now show how to determine the group of a non-leaf vertex. 
Let $v$ be a vertex not assigned to a group at maximum distance from $r$. 
So each child of $v$ is assigned to a group. 
%Say $v$ is at level $\ell$. 
Let $Y_v$ be the set of children $y$ of $v$, such that the number of children of $v$ that have the same profile as $y$ is in the range $[1,k-1]$. If $Y_v =\emptyset$ start a new singleton group $\{v\}$. 
If $Y_v\neq\emptyset$ then merge all the groups rooted at vertices in $Y_v$ into one group including $v$. 
This defines our partition of $V(T)$ into groups. Each group $X$ is \defn{rooted} at the vertex in $X$ closest to $r$ in $T$. A group $Y$ is \defn{above} a distinct group $X$ if the root of $Y$ is on the path in $T$ from the root of $X$ to $r$. 

The next claim is the key to the remainder of the proof. 

\begin{claim}
Let $uv\in E(T)$ where $u$ is the parent of $v$, and $u$ is in a different group to $v$. Then for every ranked graph $(H,L,\preceq)$ in the profile of $v$, 
the $\ell(u)$-splice of $(H,L,\preceq)$ is in the profile of $u$. 
\end{claim}

\begin{proof}
Since $(H,L,\preceq)$ is in the profile of $v$, there is an isomorphism $\phi$ from $H$ to some subgraph of $G[T^+_v]$ such that  for each vertex $x\in V(H)$ we have $L(x)=\ell(\phi(x))$, and for all distinct vertices $x,y\in V(H)$ we have that $x\prec y$ if and only if $\phi(x)$ is an ancestor of $\phi(y)$ in $T$. 

Since $u$ and $v$ are in different groups, there are $k$ children $y_1,\dots,y_k$ of $u$ (one of which is $v$) such that the profiles of $y_1,\dots,y_k$ are equal. Thus $(H,L,\preceq)$ is in the profile of each of $y_1,\dots,y_k$.
That is, for each $j\in[1,k]$, there is an isomorphism $\phi_j$ from $H$ to some subgraph of $G[T^+_{y_j}]$ such that  for each vertex $x\in V(H)$ we have $L(x)=\ell(\phi_j(x))$, and for all distinct vertices $x,y\in V(H)$ we have that $x\prec y$ if and only if $\phi_j(x)$ is an ancestor of $\phi_j(y)$ in $T$. 

Let $(H',L',\preceq')$ be the $\ell(u)$-splice of $(H,L,\preceq)$. 
We now define a function $\phi'$ from $V(H')$ to $V( G[ T^+_u] )$. 
For each vertex $(x,0)$ of $H'$ (thus with $x\in V(H)$ and $L(x)\in[0,\ell(u)]$), define $\phi'((x,0)) := \phi(x)$. 
For every other vertex $(x,j)$ of $H'$ (thus with $x\in V(H)$ and $L(x)\in[\ell(u)+1,d-1]$ and $j\in[1,k]$), define $\phi'((x,j)) := \phi_j(x)$. 

We now show that $\phi'$ is an isomorphism from $H'$ to a subgraph of $G[T^+_u]$.
Consider an edge $(x,a)(y,b)$ of $H'$. Thus $xy\in E(H)$. 
It suffices to show that  $\phi'((x,a))\phi'((y,b))\in E(G[T^+_u])$. 
%%%%%%%%%%%
First suppose that $a=b=0$. So $L(x)\in[0,\ell(u)]$ and $L(y)\in[0,\ell(u)]$. 
Thus $\phi'((x,a))=\phi(x)$ and $\phi'((y,b))=\phi(y)$. 
Since $\phi$ is an isomorphism to a subgraph of $G[ T^+_v]$, we have 
$\phi(x)\phi(y) \in E(G[ T^+_v] )$, which is a subgraph of $G[ T^+_u] $.
Hence $\phi'((x,a))\phi'((y,b)) \in E(G[ T^+_u] ) $, as desired. 
%%%%%%%%
Now suppose that $a=0$ and $b\in[1,k]$.
Thus $\phi'((x,a))=\phi(x)$ and $\phi'((y,b))=\phi_b(y)$. 
Moreover, both $\ell(\phi(x))$ and $\ell(\phi_b(x))$ equal $L(x)\in[0,\ell(u)]$. 
There is only vertex $z$ in $T^+_v$ with $\ell(z)$ equal to a specific number in $[0,\ell(u)]$. Thus $\phi'((x,a))=\phi(x)=\phi_b(x)$ ($=z$). 
Since $\phi_b$ is an isomorphism to a subgraph of $G[ T^+_{y_b}]$, we have 
$\phi_b(x)\phi_b(y) \in E(G[ T^+_{y_b}] )$, which is a subgraph of $G[ T^+_u] $.
Hence $\phi'((x,a))\phi'((y,b)) \in E(G[ T^+_u] ) $, as desired. 
%%%%%%%%%
Finally, suppose that $a=b\in[1,k]$.
Thus $\phi'((x,a))=\phi_a(x)$ and $\phi'((y,b))=\phi_b(y)=\phi_a(y)$. 
Since $\phi_a$ is an isomorphism to a subgraph of $G[ T^+_{y_a}]$, we have 
$\phi_a(x)\phi_a(y) \in E(G[ T^+_{y_a}] )$, which is a subgraph of $G[ T^+_u] $.
Hence $\phi'((x,a))\phi'((y,b)) \in E(G[ T^+_u] ) $, as desired. 
%%%%%
This shows that $\phi'$ is an isomorphism from $H'$ to a subgraph of $G[T^+_u]$.

We now verify property (A) for $(H',L',\preceq')$. 
For each vertex $(x,0)$ of $H'$ (thus with $x\in V(H)$ and $L(x)\in[0,\ell(u)]$) we have $L'((x,0))=L(x)= \ell(\phi(x)) = \ell(\phi'((x,0)))$, as desired. 
For every other vertex $(x,j)$ of $H'$ (thus with $x\in V(H)$ and $L(x)\in[\ell(u)+1,d-1]$ and $j\in[1,k]$) we have 
$L'((x,j))=L(x)= \ell(\phi_j(x)) = \ell(\phi'((x,j)))$, as desired. 
Hence property (A) is satisfied  for $(H',L',\preceq')$. 

We now verify property (B) for $(H',L',\preceq')$. 
Consider distinct vertices $(x,a),(y,b)\in V(H')$. 
First suppose that $a=0$ and $b=0$. 
Then 
$(x,a) \prec' (y,b)$ if and only if 
$x \prec y$ if and only if 
$\phi(x)$ is an ancestor of $\phi(y)$ in $T$ if and only if 
$\phi'((x,a))$ is an ancestor of $\phi'((y,b))$ in $T$, as desired. 
Now suppose that $a=0$ and $b\in[1,k]$. 
Then 
$(x,a) \prec' (y,b)$ if and only if 
$x \prec y$ if and only if 
$\phi(x)$ is an ancestor of $\phi_b(y)$ in $T$ if and only if 
$\phi'((x,a))$ is an ancestor of $\phi'((y,b))$ in $T$, as desired. 
Now suppose that $a=b\in[1,k]$. 
Then 
$(x,a) \prec' (y,b)$ if and only if 
$x \prec y$ if and only if 
$\phi_a(x)$ is an ancestor of $\phi_b(y)$ in $T$ if and only if 
$\phi'((x,a))$ is an ancestor of $\phi'((y,b))$ in $T$, as desired. 
Finally, suppose that $a,b\in[1,k]$ and $a\neq b$. 
Then $(x,a)$ and $(y,b)$ are incomparable under $\prec'$, and 
$\phi'((x,a))$ and $\phi'((y,b))$ in $T$ are unrelated in $T$, as desired. 
Hence property (B) is satisfied  for $(H',L',\preceq')$. 

So $\phi'$ is an isomorphism from $H'$ to a subgraph of $G[T^+_u]$ satisfying properties (A) and (B). Thus $(H',L',\preceq')$ is contained in $G[T^+_u]$, as desired. Since $(H,L,\preceq)$ is in the profile of $v$, we have $|V(H)| \leq (d+1)(h-1)(k+1)^{h-\ell(v)}$. Since $|V(H')|\leq (k+1) |V(H)|$ and $\ell(u)=\ell(v)-1$, we have $|V(H')| \leq  (d+1)(h-1)(k+1)^{h+1-\ell(v)} =  (d+1)(h-1)(k+1)^{h-\ell(u)}$. Thus $(H',L',\preceq')$ is in the profile of $u$. 
\end{proof}

The proof now divides into two cases. If some group $X_0$ is adjacent in $G$ to at least $h-1$ other groups above $X_0$, then we show that $G$ contains $\WeakClosure{h,k}$ as a minor. Otherwise, every group $X$ is adjacent in $G$ to at most $h-2$ other groups above $X$, in which case we show that $G$ is $(h-1)$-colourable with bounded clustering. 

\subsubsection*{Finding the Minor}

Suppose that some group $X_0$ is adjacent in $G$ to at least $h-1$ other groups $X_1,\dots,X_{h-1}$ above $X_0$. We now show that $G$ contains $\WeakClosure{h,k}$ as a minor; refer to \cref{ConstructMinor}. 
For $i\in[1,h-1]$, since $X_i$ is above $X_0$, 
the root $v_i$ of $X_i$ is on the $v_0r$-path in $T$. 
Without loss of generality, 
$v_0,v_1,\dots,v_{h-1}$ appear in this order on the $v_0r$-path in $T$. 
For $i\in[1,h-1]$, let $w_i$ be a vertex in $X_i$ adjacent to some vertex $z_i$ in $X_0$; since $G$ is a subgraph of the closure of $T$, $w_i$ is on the $v_0r$-path in $T$. For $i\in[0,h-2]$, let $u_i$ be the parent of $v_i$ in $T$ (which exists since $v_{h-2}\neq r$). 
%\referee{Page 9, Line 4: You might have to assume that $v_i$ is a descendant of $v_{i+1}$ for every $i \in [0 , h-2]$.}
So $u_i$ is not in $X_i$ (but may be in $X_{i+1}$). Note that $v_0,u_0,w_1,v_1,u_1,\dots,w_{h-2},v_{h-2},u_{h-2},w_{h-1},v_{h-1}$ appear in this order on the $v_0r$-path in $T$, where $v_0,v_1,\dots,v_{h-1}$ are distinct (since they are in distinct groups). 

\begin{figure}[h]
\centering\includegraphics[width=\textwidth]{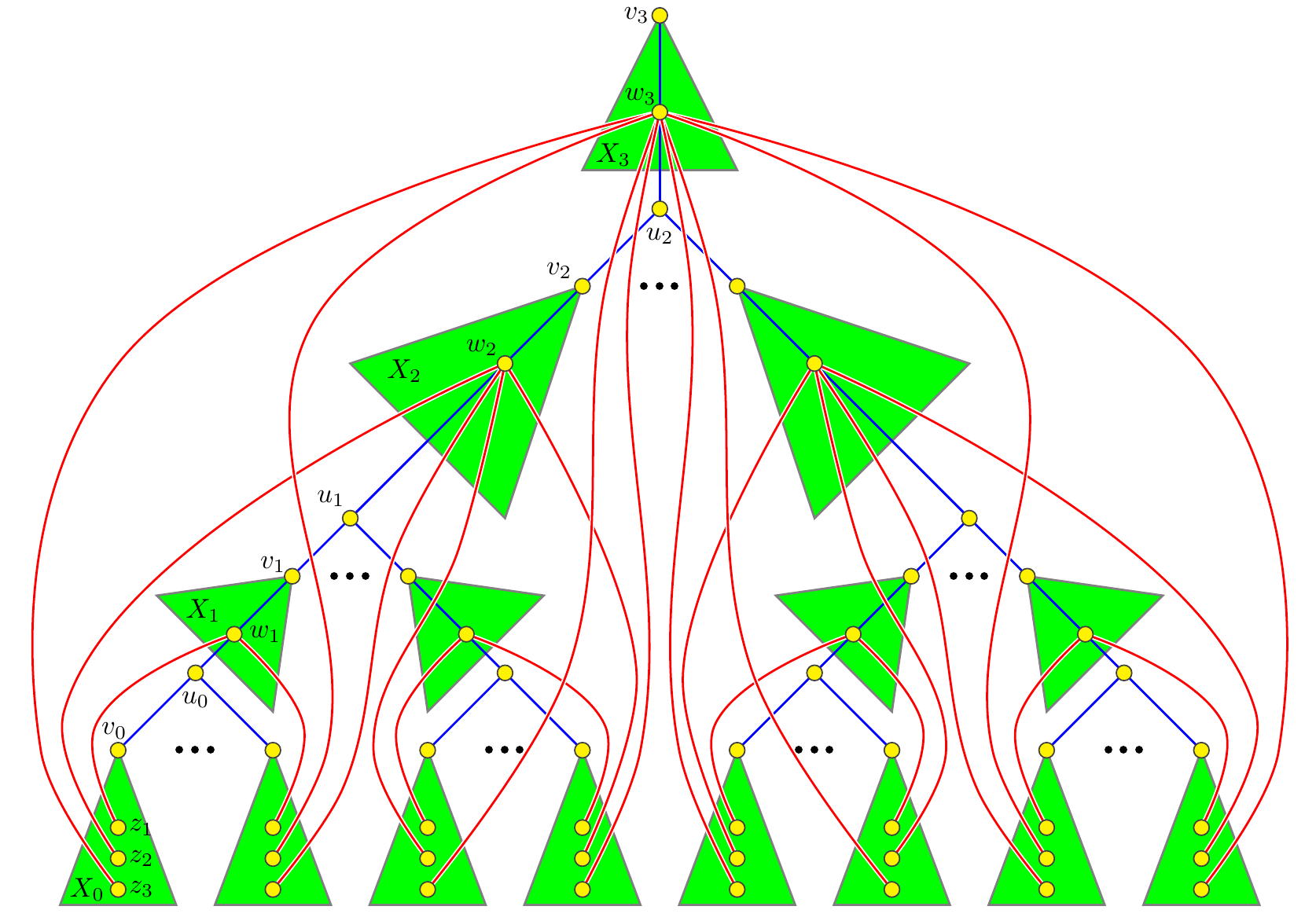}
\caption{Construction of a \WeakClosure{4,k} minor (where $u_i$ might be in $X_{i+1}$). 
\label{ConstructMinor}}
\end{figure}

Let $P_j$ be the $z_jr$-path in $T$ for $j\in[1,h-1]$. 
Let $H_0$ be the graph with $V(H_0):=V( P_1\cup\dots\cup P_{h-1} )$ and $E(H_0) := \{z_jw_j:j\in[1,h-1]\}$. Define the function $L_0:V(H_0)\to[0,d-1]$ by $L_0(x):=\ell(x)$ for each $x\in V(H_0)$. Define the partial order $\preceq_0$ on $V(H_0)$, where $x\prec_0 y$ if and only if $x$ is ancestor of $y$ in $T$. Thus $(H_0,L_0,\preceq_0)$ is a ranked graph. By construction, 
$(H_0,L_0,\preceq_0)$ is contained in $G[T^+_{v_0}]$. Since $H_0$ has less than $(d+1)(h-1)$ vertices, $H_0$ is in the profile of $v_0$. For $i=0,1,\dots,h-2$, let $(H_{i+1},L_{i+1},\prec_{i+1})$ be the $\ell(u_i)$-splice of $(H_i,L_i,\prec_i)$. 

By induction on $i$, using Claim~1 at each step and since $G[T^+_{u_i}]\subseteq G[T^+_{v_{i+1}}]$, we conclude that for each $i\in[0,h-1]$, the ranked graph $(H_i,L_i,\preceq_i)$ is in the profile of $v_i$. 
In particular, $(H_{h-1},L_{h-1},\prec_{h-1})$ is in the profile of $v_{h-1}$, and $H_{h-1}$  is isomorphic to a subgraph of $G$. Note that each vertex of $H_{h-1}$ is of the form $(((\dots(x,d_1),d_2),\dots),d_{h-1})$ for some $x\in V(H_0)$ and $d_1,\dots,d_{h-1}\in[0,k]$. For brevity, call such a vertex $x\langle{d_1,\dots,d_{h-1}\rangle}$. Note that if $x=w_j$ for some $j\in[1,h-1]$, then $d_1=\dots=d_j=0$ (since $w_j$ is above $u_i$ whenever $i<j$, and $(H_{i+1},L_{i+1},\prec_{i+1})$ is the $\ell(u_i)$-splice of $(H_i,L_i,\preceq_i)$). 

For $x\in V(H_0)$, let $\Lambda_x$ be the set of vertices $x\langle{d_1,\dots,d_{h-1}\rangle}$ in $H_{h-1}$. By construction, no two vertices in $\Lambda_x$ are comparable under $\preceq_{h-1}$.
Therefore, by property (B), $V(T_a)\cap V(T_b)=\emptyset$ for all distinct $a,b\in \Lambda_x$. 
In particular, $V(T_{a})\cap V(T_b)=\emptyset$ for all distinct $a,b\in \Lambda_{v_0}$. As proved above, $G[T_a] $ is connected for each $a\in V(T)$. Let $G'$ be the graph obtained from $G$ by contracting $G[T_a]$ into a single vertex $\alpha\langle{d_1,\dots,d_{h-1}\rangle}$, for each $a=v_0\langle{d_1,\dots,d_{h-1}\rangle} \in \Lambda_{v_0}$. So $G'$ is a minor of $G$. 
	
Let $U$ be the tree with vertex set 
\begin{equation*}
\{ \langle{ d_1,\dots,d_{h-1} \rangle}: \exists j\in[0,h-1] \;d_1=\dots=d_j=0\text{ and }d_{j+1},\dots,d_{h-1} \in[1,k] \},
\end{equation*}
where the parent of $(0,\dots,0,d_{j+1},d_{j+2},\dots,d_{h-1})$ is  $(0,\dots,0,d_{j+2},\dots,d_{h-1})$. Then $U$ is isomorphic to the complete $k$-tree of height $h$ rooted at $\langle 0,\dots,0\rangle$. 
We now show that the weak closure of $U$ is a subgraph of $G$', where 
each vertex $\langle 0,\dots,0,d_{j+1},\dots,d_{h-1}\rangle$ of $U$ with $j\in[1,h-1]$ is mapped to vertex $w_j\langle 0,\dots,0,d_{j+1},\dots,d_{h-1}\rangle$ of $G'$, and 
each other vertex $\langle d_1,\dots,d_{h-1}\rangle$ of $U$ is mapped to $\alpha\langle d_1,\dots,d_{h-1}\rangle$ of $G'$. 
For all $d_1,\dots,d_{h-1}\in[1,k]$ and $j\in[1,h-1]$ the vertex $z_j\langle{d_1,\dots,d_{h-1}\rangle}$ of $G$ is contracted into the vertex 
$\alpha\langle{d_1,\dots,d_{h-1}\rangle}$ of $G'$. 
By construction, 
$z_j\langle{d_1,\dots,d_{h-1}\rangle}$ is adjacent to 
$w_j\langle{0,\dots,0,d_{j+1},\dots,d_{h-1}\rangle}$ in $G$. 
So $\alpha\langle{d_1,\dots,d_{h-1}\rangle}$ is adjacent to 
$w_j\langle{0,\dots,0,d_{j+1},\dots,d_{h-1}\rangle}$ in $G'$. 
This implies that the weak closure of $U$ (that is, \WeakClosure{h,k}) is isomorphic to a subgraph of $G$', and is therefore a minor of $G$.

\subsubsection*{Finding the Colouring}

Now assume that every group $X$ is adjacent in $G$ to at most $h-2$ other groups above $X$. Then $(h-1)$-colour the groups in order of distance from the root, such that every group $X$ is assigned a colour different from the colours assigned to the neighbouring groups above $X$. Assign each vertex within a group the same colour as that assigned to the whole group. This defines an $(h-1)$-colouring of $G$.

Consider the function $s:[0,d-1]\to\bb{N}$ recursively defined by 
\begin{equation*}
s(\ell) := \begin{cases}
1 & \text{ if }\ell=d-1\\
(k-1) \cdot M_{\ell+1} \cdot s(\ell+1) & \text{ if } \ell\in[0,d-2].
\end{cases}
\end{equation*}
%\referee{Page 10, Line -10: It seems to me that $s(0)$ also depends on $h$, as the deﬁnition of $M_\ell$ depends on $N_\ell$ which depends on $h$.}
Then every group at level $\ell$ has at most $s(\ell)$ vertices. By construction, our $(h-1)$-colouring of $G$ has clustering $s(0)$, which is bounded by a function of $d$, $k$ and $h$, as desired. 
\end{proof}

%%%%%%%%%%%%%%%%%%%%%%%%%%%%%%%%%%%%%%%%%%%%%%%%%%%
\section{Pathwidth}
\label{Pathwidth}

The following lemma of independent interest is the key to proving \cref{MainMain}. Note that \citet{Eppstein20} independently discovered the same result (with a slightly weaker bound on the path length). The decomposition method in the proof has been previously used, for example,  by \citet[Lemma~17]{DJKW16}. 

\begin{lem}
\label{Pathwidth2ColourPathlength}
Every graph with pathwidth at most $w$ has a vertex 2-colouring such that each monochromatic path has at most $(w+3)^w$ vertices. 
\end{lem}

\begin{proof}
We proceed by induction on $w\geq 1$. Every graph with pathwidth 1 is a caterpillar, and is thus properly 2-colourable.  Now assume $w\geq 2$ and the result holds for graphs with pathwidth at most $w-1$. 
Let $G$ be a graph with pathwidth at most $w$. 
Let $(B_1,\dots,B_n)$ be a path-decomposition of $G$ with width at most $w$. 
Let $t_1,t_2,\dots,t_m$ be a maximal sequence such that $t_1=1$ and 
for each $i\geq 2$, $t_i$ is the minimum integer such that 
$B_{t_i} \cap B_{t_{i-1}}=\emptyset$.
For odd $i$, colour every vertex in $B_{t_i}$ `red'. 
For even $i$, colour every vertex in $B_{t_i}$ `blue'. 
Since $B_{t_i} \cap B_{t_{i-1}}=\emptyset$ for $i\geq 2$, no vertex is coloured twice. 
Let $G'$ be the subgraph of $G$ induced by the uncoloured vertices. 
By the choice of $B_{t_i}$,  for $i\geq 2$ each bag $B_j$ with $j\in[t_{i-1}+1,t_i-1]$ intersects $B_{t_{i-1}}$. 
Thus $(B_1\cap V(G'),\dots,B_n\cap V(G'))$ is a path-decomposition of $G'$ of width at most $w-1$. 
By induction, $G'$ has a vertex 2-colouring such that each monochromatic path has at most $(w+3)^{w-1}$ vertices. 
Since $B_{t_i}\cup B_{t_{i+2}}$ separates  $B_{t_i+1}\cup\dots\cup B_{t_{i+2}-1}$ from the rest of $G$,
each monochromatic component of $G$ is contained in 
$B_{t_i+1}\cup\dots\cup B_{t_{i+2}-1}$ for some $i\in[0,n-2]$. 
Consider a monochromatic path $P$ in $G[ B_{t_i+1}\cup\dots\cup B_{t_{i+2}-1} ] $. 
Then $P$ has at most $w+1$ vertices in $B_{t_{i+1}}$. 
Note that $P- B_{t_{i+1}}$ is contained in $G'$. 
Thus $P$ consists of up to $w+2$ monochromatic subpaths in $G'$ plus $w+1$ vertices in $B_{t_{i+1}}$.
Hence $P$ has at most $(w+2) (w+3)^{w-1} + (w+1) < (w+3)^{w}$  vertices. 
\end{proof}

\citet{Sparsity} showed  that if a graph $G$ contains no path on $k$ vertices, then $\td(G)<k$ (since $G$ is a subgraph of the closure of a DFS spanning tree with height at most $k$). 
%Let $G_1,\dots,G_n$ be the connected components of $G$. 
%Let $T_i$ be a DFS spanning tree of $G_i$. 
%So $G_i$ is a subgraph of the closure of $T_i$. 
%Note that $T_i$ has height less than $k$. 
%So $G$ has $\td(G) <k$ and $\ctd(G) \leq k$. 
Thus \cref{Pathwidth2ColourPathlength} implies: 

\begin{cor}
\label{PathwidthTreedepth}
Every graph with pathwidth at most $w$ has a vertex 2-colouring such that each monochromatic component has treedepth at most $(w+3)^w$.
\end{cor}

\begin{proof}[Proof of \cref{MainMain}] 
Let $\GG$ be a minor-closed class of graphs, each with pathwidth at most $w$. 
Let $h$ be the minimum integer such that $\Closure{h,k} \not\in \GG$ for some $k\in\mathbb{N}$. 
Consider $G\in\GG$. 
Thus \WeakClosure{h,k+1} is not a minor of $G$ (since \Closure{h,k} is a minor of \WeakClosure{h,k+1}, as noted above). 
By \cref{PathwidthTreedepth}, $G$ has a vertex 2-colouring such that each monochromatic component $H$ of $G$ has treedepth at most $(w+3)^w$. 
Thus \WeakClosure{h,k+1} is not a minor of $H$.
By \cref{UpperBound}, 
$H$ is $(h-1)$-colourable with clustering $c((w+3)^w,k+1,h)$. 
Taking a product colouring, $G$ is $(2h-2)$-colourable with clustering $c((w+3)^w,k+1,h)$. 
Hence $ \dchi(\GG) \leq \cchi(\GG) \leq 2h-2$. 
\end{proof}

Note that \cref{Pathwidth2ColourPathlength} cannot be extended to the setting of bounded tree-width graphs: Esperet and Joret (see \citep[Theorem~4.1]{LO17}) proved that for all positive integers $w$ and $d$ there exists a graph $G$ with tree-width at most $w$ such that for every $w$-colouring of $G$ there exists a monochromatic component of $G$ with diameter greater than $d$ (and thus with a monochromatic path on more than $d$ vertices, and thus with treedepth at least $\log_2 d$). 

%%%%%%%%%%%%%%%%%%%%%%%%%%%%%%%%%%
\section{Fractional Colouring}
\label{FractionalColouring}

This section proves \cref{FractionalChromaticNumber}. The starting point is the following key result of 
\citet{DS20}.\footnote{\citet{DS20} expressed their result in the terms of ``treedepth fragility''. The sentence ``proper minor-closed classes are fractionally treedepth-fragile'' after Theorem~31 in \cite{DS20} is equivalent to \cref{t:DS}. Informally speaking, \cref{t:DS} shows that the fractional ``treedepth'' chromatic number of every minor-closed class equals 1.}

\begin{thm}[\citep{DS20}]\label{t:DS}
	For every proper minor-closed class $\GG$ and every $\delta > 0$ there exists $d \in \bb{N}$ satisfying the following. For every $G \in \GG$ there exist $s \in \bb{N}$ and $X_1,X_2, \ldots, X_s \subseteq V(G)$ such that:  
	\begin{itemize}
		\item $\td(G[X_i]) \leq d$, and
		\item every $v \in V(G)$ belongs to at least $(1 - \delta)s$ of these sets.
	\end{itemize} 
\end{thm}

We now prove a lower bound on the fractional defective chromatic number of the closure of complete trees of given height. 

\begin{lem}\label{l:lower}
	Let $\mc{C}_h := \{C \la h,k\ra\}_{k \in \bb{N}}$. Then $\dfchi(\mc{C}_h) \geq h$.
\end{lem}

\begin{proof}
	We show by induction on $h$ that if $C \la h,k\ra$ is  fractionally $t$-colourable with defect $d$,  then $t\geq h - (h-1)d/k$.  This clearly implies the lemma. The base case $h=1$ is trivial. 
	
	For the induction step, suppose that $G:=C \la h,k\ra$ is  fractionally $t$-colourable with defect $d$. Thus there exist $Y_1,Y_2, \ldots, Y_s \subseteq V(G)$ and $\alpha_1,\ldots,\alpha_s \in [0,1]$ such that:
	\begin{itemize}
	\item every component of $G[Y_i]$ has maximum degree at most $d$, 
	\item $\sum_{i=1}^s \alpha_i \leq t$, and 
	\item $\sum_{i : v \in Y_i}\alpha_i \geq 1$ for every $v \in V(G)$.
	\end{itemize} 
	
	Let $r$ be the vertex of $G$ corresponding to the root of the complete $k$-ary tree and let $H_1,\ldots,H_k$ be the components of $G - r$. Then each $H_i$ is isomorphic to $C \la h-1,k\ra$. Let $J_0 := \{j : r \in Y_j \}$, and let $J_i := \{j : Y_j \cap V(H_i) \neq\emptyset \}$ for $i\in[1,k]$. Denote $\sum_{j \in J_i} \alpha_j$ by $\alpha(J_i)$ for brevity. Thus $\alpha(J_0) \geq 1$. For $i\in[1,k]$, the subgraph $H_i$ is $\alpha(J_i)$-colourable with defect $d$, and thus $\alpha(J_i) \geq h-1 - (h-2)d/k$  by the induction hypothesis. Thus 
	$$ (k-d)\alpha(J_0) + \sum_{i=1}^{k}\alpha(J_i) \geq (k-d)+k(h-1) - (h-2)d = kh-(h-1)d.$$
	If $j \in J_0$ then $Y_j$ intersects at most $d$ of $H_1,\dots,H_k$ (since $G[Y_j]$ has maximum degree at most $d$). Thus every $\alpha_j$ appears with coefficient at most $k$ in the left side of the above inequality, implying
	$$ (k-d)\alpha(J_0) + \sum_{i=1}^{k}\alpha(J_i) \leq k \sum_{i=1}^s \alpha_i \leq kt.$$   
	Combining the above inequalities yields the claimed bound on $t$. 
\end{proof}    

\begin{proof}[Proof of \cref{FractionalChromaticNumber}]
By \cref{l:lower}, $$\cfchi(\GG)  \geq  \dfchi(\GG) \geq \tcn(\GG) -1.$$
It remains to show that $\cfchi(\GG)  \leq \tcn(\GG) -1$. 
Equivalently, we need to show that for all $h,k \in \bb{N}$ and $\eps >0$, if $\Closure{h,k} \not\in \GG$ then  there exists $c$ such that every graph in $\GG$ is fractionally $(h-1+\eps)$-colourable with clustering $c$. This is trivial for $h=1$, and so we assume $h \geq 2$.
	
	Let $d \in \bb{N}$ satisfy the conclusion of \cref{t:DS} for the class $\GG$  and $\delta = 1 - \frac{1}{1+\eps/(h-1)}$. Choose $c=c(d,k+1,h)$ to satisfy the conclusion of \cref{UpperBound}. 
%	\referee{Page 13, Line 9: $c$ might have to depend on $h$.}
	We show that $c$ is as desired.
	
	Consider $G \in \GG$. By the choice of $d$ there exists  $s \in \bb{N}$ and  $X_1,X_2, \ldots, X_s \subseteq V(G)$ such that:
	\begin{itemize}
		\item $\td(G[X_i]) \leq d$, and
		\item every $v \in V(G)$ belongs to at least $(1 - \delta)s$ of these sets.
	\end{itemize} 
Since $\Closure{h,k} \not\in \GG$, we have $\WeakClosure{h,k+1} \not\in \GG$, and by the choice of $c$, for each $i\in[1,s]$ there exists a partition $(Y^1_i,Y^2_i,\ldots,Y^{h-1}_i)$ of $X_i$ such that every component of $G[Y^j_i]$ has at most $c$ vertices. Every vertex of $G$ belongs to at least $(1 - \delta)s$ sets $Y^{j}_i$ where $i\in[1,s]$ and $j\in[1,h-1]$. Considering these sets with equal coefficients $\alpha^{j}_i := \frac{1}{(1 - \delta)s}$, we conclude that $G$ is fractionally $\frac{h-1}{1-\delta}$-colourable with clustering $c$, as desired (since $\frac{h-1}{1-\delta}=h-1+\eps$).
\end{proof}  

\subsection*{Acknowledgement}

This work was partially completed while SN was visiting Monash University supported by a Robert Bartnik Visiting Fellowship. SN thanks the School of Mathematics at Monash University for its hospitality. Thanks to the referee for several helpful comments. 

%\comment{Is the following extension of \cref{Pathwidth2ColourPathlength} true? 
%Let $G$ be a graph that has a tree-decomposition $(T,(B_x:x\in V(T)))$ of width $k$, such that $T$ has pathwidth at most $c$. Then $G$ has a vertex $(c+1)$-colouring such that each monochromatic path has at most $f(c,k)$ vertices. \\
%YES $G$ has pathwidth at most $(c+1)(k+1)-1$
%}

%%%  Squashing the bibliography 
{
\renewcommand{\baselinestretch}{1}
  \let\oldthebibliography=\thebibliography
  \let\endoldthebibliography=\endthebibliography
  \renewenvironment{thebibliography}[1]{%
    \begin{oldthebibliography}{#1}%
      \setlength{\parskip}{0ex}%
      \setlength{\itemsep}{0ex}}%
  {\end{oldthebibliography}}

%\bibliographystyle{../../BibTex/DavidNatbibStyle}
%\bibliography{../../BibTex/myBibliography}

\def\soft#1{\leavevmode\setbox0=\hbox{h}\dimen7=\ht0\advance \dimen7
	by-1ex\relax\if t#1\relax\rlap{\raise.6\dimen7
		\hbox{\kern.3ex\char'47}}#1\relax\else\if T#1\relax
	\rlap{\raise.5\dimen7\hbox{\kern1.3ex\char'47}}#1\relax \else\if
	d#1\relax\rlap{\raise.5\dimen7\hbox{\kern.9ex \char'47}}#1\relax\else\if
	D#1\relax\rlap{\raise.5\dimen7 \hbox{\kern1.4ex\char'47}}#1\relax\else\if
	l#1\relax \rlap{\raise.5\dimen7\hbox{\kern.4ex\char'47}}#1\relax \else\if
	L#1\relax\rlap{\raise.5\dimen7\hbox{\kern.7ex
			\char'47}}#1\relax\else\message{accent \string\soft \space #1 not
		defined!}#1\relax\fi\fi\fi\fi\fi\fi}

\end{document}